\documentclass{article}

\usepackage{amsmath}
\usepackage{amsthm}
\usepackage{amsfonts}
\usepackage{amsbsy}
\usepackage{amssymb}
\usepackage[initials]{amsrefs}

\newtheorem{theorem}{Theorem}

\newtheorem{lemma}{Lemma}

\newcommand{\R}{{\mathbb R}}
\newcommand{\Z}{{\mathbb Z}}
\newcommand{\C}{{\mathbb C}}
\newcommand{\set}[2]{ \left\{ #1 \ \left| \ #2 \right. \right\} }

\usepackage[pdf]{pstricks}

\newcommand{\ang}[1]{\left< #1 \right>}
\newcommand{\M}{{\mathcal M}}
\newcommand{\at}[1]{\left. #1 \right|}
\newcommand{\F}{{\mathcal F}}
\newtheorem*{example}{Example}
\newcommand{\SL}{{\mathrm{SL}}}

\newcommand{\aff}{\mu_{\!_{\mathcal A}}}

\title{On the Oberlin affine curvature condition}
\date{\today}
\author{Philip T. Gressman\footnote{Partially supported by NSF grant DMS-1361697.}}

\begin{document}
\maketitle

\begin{abstract}
In this paper we generalize the well-known notions of affine arclength and affine hypersurface measure to submanifolds of any dimension $d$ in $\R^n$, $1 \leq d \leq  n - 1$. We show that a canonical affine invariant measure exists and that, modulo sufficient regularity assumptions on the submanifold, the measure satisfies the affine curvature condition of D.~Oberlin with an exponent which is best possible. The proof combines aspects of Geometric Invariant Theory, convex geometry, and frame theory. A significant new element of the proof is a generalization to higher dimensions of an earlier result \cite{gressman2009} concerning inequalities of reverse Sobolev type for polynomials on arbitrary measurable subsets of the real line.
\end{abstract}

\section{Introduction}
Many of the deep questions in harmonic analysis, including Fourier restriction, decoupling theory, or $L^p$-improving estimates for geometric averages, deal with certain operators associated to submanifolds of Euclidean space. In most cases, the ``nicest possible'' submanifolds are, informally, as far as possible from lying in any affine hyperplane. Many of these problems also exhibit natural affine invariance, meaning that when the underlying Euclidean space is transformed by a measure-preserving affine linear mapping, the the relevant quantities (norms, etc.) are unchanged. This simple observation leads naturally to the question of how in general to properly quantify this sort of well-curvedness in a way that respects affine invariance. Of the many approaches to this question, one particularly successful strategy has been the use of the so-called affine arclength measure for curves and the analogous notion of affine hypersurface measures (sometimes called the equiaffine measure). In the former case, affine arclength is defined on a curve parametrized by $\gamma : I \rightarrow \R^n$ by 
\[ \int f d \aff := \int_{I} f(\gamma(t)) \left| \det (\gamma'(t),\ldots,\gamma^{(n)}(t))\right|^{\frac{2}{n(n+1)}} dt, \]
while equiaffine measure on the graph $(x,\varphi(x))$ over $U \subset \R^{n-1}$ is given by
\[ \int f d \aff := \int_U f(x,\varphi(x)) |\det \nabla^2 \varphi(x)|^{\frac{1}{n+1}} dx, \]
where $\nabla^2 \varphi$ is the Hessian matrix of second derivatives of $\varphi$.
Though these measures were well-known outside harmonic analysis for quite some time (see, for example, \cites{guggenheimer1963,lutwak1991}), their first appearances within the field are somewhat more recent, in work of S\"{o}lin \cite{sjolin1974} (in two dimensions, generalized later by Drury and Marshall \cite{dm1985}) and Carbery and Ziesler \cite{cz2002}, respectively. Both measures have the property that they are independent of the parametrization and that they are unchanged when the curve or surface is transformed by a measure-preserving affine mapping. These measures and certain ``variable coefficient'' generalizations to families of curves and hypersurfaces have played a central role in the Fourier restriction problem as well as the problem of characterizing the $L^p$--$L^q$ mapping properties of geometrically-constructed convolution operators, two problems which have been of sustained interest for many years
\cites{drury1990,choi1999,oberlin1999II,dw2010,stovall2016,dlw2009,stovall2014,gressman2013,oberlin2012}.

The deep connections between analysis and geometry enjoyed by affine arclength and hypersurface measures naturally lead to the problem of generalizing these objects to manifolds of arbitrary dimension or even to abstract measure-theoretic settings. One particularly interesting approach is due to D.~Oberlin \cite{oberlin2000II} (which generalizes an earlier observation of Graham, Hare, and Ritter \cite{ghr1989} in one dimension), who introduced the following condition on nonnegative measures $\mu$ associated to submanifolds: a measure $\mu$ on a $d$-dimensional immersed submanifold of $\R^n$ will be said to satisfy the Oberlin condition with exponent $\alpha > 0$ when there exists a finite positive constant $C$ such that for every $K$ in the set $\mathcal K_n$ of compact convex subsets of $\R^n$,
\begin{equation}
\mu(K) \leq C |K|^\alpha, \label{oberlin}
\end{equation}
where $|K|$ represents the usual Lebesgue measure of $K$ in $\R^n$. When restricted to the class of balls with respect to the standard metric on $\R^n$, the condition \eqref{oberlin} becomes a familiar inequality from geometric measure theory. Unlike in that setting, here the exponent $\alpha$ measures not just dimension of the measure, but also a certain kind of curvature (for the simple reason that \eqref{oberlin} cannot hold for any $\alpha > 0$ when $\mu$ is supported on a hyperplane, which can be seen by taking $K$ to be increasingly thin in the direction transverse to such a hyperplane). Oberlin observed that this condition is necessary for Fourier restriction or $L^p$-improving estimates to hold; in particular,
\[ \ \left( \int |\hat f|^q d \mu \right)^\frac{1}{q} \lesssim ||f||_{L^p(\R^n)} \ \forall f \in L^p(\R^n) \Rightarrow \mu(K) \lesssim |K|^{\frac{q}{p'}} \ \forall K \in {\mathcal K}_n\]
and
\[ || f * \mu ||_{L^q(\R^n)} \lesssim ||f||_{L^p(\R^n)} \ \forall f \in L^p(\R^n) \Rightarrow \mu(K) \lesssim |K|^{\frac{1}{p} - \frac{1}{q}} \ \forall K \in {\mathcal K}_n, \]
where $\hat{\cdot}$ is the Fourier transform, $*$ is convolution. Here and throughout the paper, the notation $A \lesssim B$ means that there is a finite positive constant $C$ such that $A \leq C B$ and this constant $C$ is independent of the relevant variables (functions and sets in this case) appearing in the expressions or quantities $A$ and $B$. By virtue of known results for these two problems, the affine arclength and hypersurface measures must satisfy \eqref{oberlin} for appropriate exponents $\alpha$ when suitable regularity hypotheses on the submanifolds are imposed.  

The significance of the Oberlin condition \eqref{oberlin} for curves and hypersurfaces in $\R^n$ is that, up to a constant factor, the affine arclength and affine hypersurface measures on an immersed submanifold are the unique largest measures on the manifolds satisfying \eqref{oberlin} when $\alpha = 2 / (n^2+n)$ and $\alpha = (n-1)/(n+1)$, respectively. More precisely, in the case of hypersurfaces (first established by Oberlin \cite{oberlin2000II}), if $\mu$ is any nonnegative measure on an immersed hypersurface $\mathcal M \subset \R^n$, if $\mu$ satisfies \eqref{oberlin} with $\alpha = (n+1)/(n-1)$, then $\mu \lesssim \aff$ for affine hypersurface measure $\aff$ (where $\lesssim$ here means $\mu(E) \lesssim \aff(E)$ uniformly for all Borel sets $E$). Moreover, subject to certain algebraic limits on the complexity of the immersion, $\aff$ itself satisfies \eqref{oberlin} for this same exponent. The condition \eqref{oberlin} also turns out to be equivalent to the boundedness of certain geometrically-constructed multilinear determinant functionals \cite{gressman2010} and leads to a natural affine generalization of the classical Hausdorff measure \cite{oberlin2003}.
 
This paper examines the Oberlin condition for arbitrary $d$-dimensional submanifolds of $\R^n$ (where $1 \leq d \leq n$) and characterizes it in the case of maximal nondegeneracy. Specifically, the analogous results to those just mentioned above are established in all dimensions and codimensions: an affine invariant measure is constructed which is essentially the largest-possible measure satisfying the Oberlin condition for the largest nontrivial choice of $\alpha$.
To say that $\alpha$ is nontrivial means simply that there is a nonzero measure, satisfying \eqref{oberlin} for this $\alpha$, on some immersed submanifold of the given dimension and codimension. As in the case of curves and hypersurfaces, the largest nontrivial $\alpha$ can be understood as a ratio of the intrinsic dimension of the submanifold and its ``homogeneous dimension,'' which captures information about scaling and curvature-like properties to be measured. The correct value of homogeneous dimension is defined as follows: when $d$ and $n$ are fixed, let the homogeneous dimension $Q$ be defined to be the smallest positive integer which equals the sum of the degrees of some collection of $n$ distinct, nonconstant monomials in $d$ variables (see Figure \ref{thefig}). The main result of this paper is Theorem \ref{mainthm}:
\begin{theorem}
Suppose $\mathcal M$ is an immersed $d$-dimensional submanifold of $\R^n$ equipped with a nonnegative measure $\mu$. \label{mainthm} Then the following are true:
\begin{enumerate}
\item  \label{ocfails} If \eqref{oberlin} holds for $\mu$ with exponent $\alpha > d/Q$, then $\mu$ is the zero measure. 
\item There is a nonnegative measure $\aff$ on $\mathcal M$ such that if \eqref{oberlin} holds for $\mu$ when $\alpha = d/Q$, then $\mu \lesssim \aff$. \label{ocub}
\item Under certain algebraic constraints on the immersion of $\mathcal M$ in $\R^d$ (which are satisfied globally for polynomial embeddings and locally for real analytic embeddings), $\aff$ satisfies the Oberlin condition \eqref{oberlin} with $\alpha = d/Q$. \label{ocineq}
\item \label{nontriv} There is a $d$-dimensional submanifold $\M$ of $\R^n$ for which $\aff$ has everywhere strictly positive density with respect to Lebesgue measure on $\M$ and satisfies \eqref{oberlin} with $\alpha = d/Q$.
\end{enumerate}
\end{theorem}
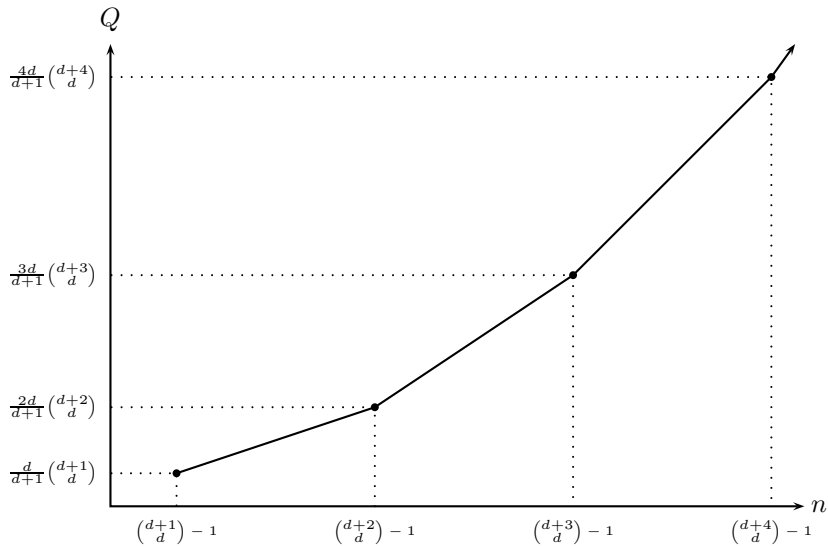
\begin{figure}
\begin{center}
\psset{unit=50pt}
\begin{pspicture}(0,0.25)(5.2,3.75)
\uput[90](0,3.75){$Q$} \uput[0](5.2,0.25){$n$}
\uput[-90](0.5,0.25){\tiny $\binom{d+1}{d}-1$} \uput[-90](2,0.25){\tiny $\binom{d+2}{d}-1$} \uput[-90](3.5,0.25){\tiny $\binom{d+3}{d} - 1$} \uput[-90](5,0.25){\tiny $\binom{d+4}{d}-1$}
\uput[-180](0,0.5){\tiny $\frac{d}{d+1} \binom{d+1}{d}$}
\uput[-180](0,1){\tiny $\frac{2d}{d+1} \binom{d+2}{d}$}
\uput[-180](0,2){\tiny $\frac{3d}{d+1} \binom{d+3}{d}$}
\uput[-180](0,3.5){\tiny $\frac{4d}{d+1} \binom{d+4}{d}$}
\psline[arrows=<->](0,3.75)(0,0.25)(5.25,0.25)
\psline[arrows=->](0.5,0.5)(2,1)(3.5,2)(5,3.5)(5.18,3.75)
\pscircle*(0.5,0.5){0.03} \pscircle*(2,1){0.03} \pscircle*(3.5,2){0.03} \pscircle*(5,3.5){0.03}
\psline[linestyle=dotted](0,0.5)(0.5,0.5)(0.5,0.25)
\psline[linestyle=dotted](0,1)(2,1)(2,0.25)
\psline[linestyle=dotted](0,2)(3.5,2)(3.5,0.25)
\psline[linestyle=dotted](0,3.5)(5,3.5)(5,0.25)
\end{pspicture}
\end{center}
\caption{This plot shows the homogeneous dimension $Q$ as a function of $n$ for $d$ fixed. The graph is piecewise linear with slope $k+1$ from the point $(\binom{d+k}{d}-1,\frac{k d}{d+1} \binom{d+k}{d})$ to the point $(\binom{d+k+1}{d}-1,\frac{(k+1) d}{d+1} \binom{d+k+1}{d })$ for each $k \geq 1$.} \label{thefig}
\end{figure}
Condition \ref{ocub} shows that the measure $\aff$ to be constructed (which is intrinsic, invariant under measure-preserving affine linear transformations of $\R^n$, and agrees up to normalization with affine arclength and equiaffine measure when $d=1,n-1$, respectively) is, up to a constant factor, the unique maximal measure on $\M$ which satisfies \eqref{oberlin} for $\alpha = d/Q$. This extends the result of Oberlin for equiaffine measure \cite{oberlin2000II} to submanifolds of any dimension.

The structure of the rest of this paper is as follows. In Section \ref{affinesec}, the measure $\aff$ is constructed by combining ideas of Kempf and Ness \cite{kn1979} from Geometric Invariant Theory together with a simple but far-reaching observation that any covariant tensor field on a manifold can be used to construct an associated measure on that manifold in a way that generalizes the relationship between the Riemanninan metric tensor and the Riemannian volume. In particular, the measure $\aff$ will be the measure associated to an ``affine curvature tensor'' on the manifold $\mathcal M$ immersed in $\R^n$. After these constructions are complete, Parts \ref{ocfails} and \ref{ocub} of Theorem \ref{mainthm} follow in a rather immediate way.

Section \ref{suffsec} is devoted to the proof of Parts \ref{ocineq} and \ref{nontriv} of Theorem 1. Part \ref{ocineq} is proved by first generalizing Theorem 1 of \cite{gressman2009} to higher dimensions. The result, Lemma \ref{derivlemm}, is interesting in its own right and will have important implications for the theory of $L^p$-improving estimates for averages over submanifolds in much the same way that Theorem 1 of \cite{gressman2009} formed the basis for a new proof of a restricted version Tao and Wright's result \cite{tw2003} for averages over curves. The final part of Section \ref{suffsec} shows that the measure $\aff$ is not trivial by constructing submanifolds on which it is possible to say with certainty that the density of $\aff$ with respect to Lebesgue measure is never zero.
Finally, Section \ref{appendix} establishes uniform estimates for the number of nondegenerate solutions of certain systems of equations. These estimates are important for Part \ref{ocineq} of Theorem \ref{mainthm}.

\section{Affine geometry and necessity}
\label{affinesec}
\subsection{Geometric Invariant Theory}
The main ideas and results from Geometric Invariant Theory that will be used in this paper come from the seminal work of Kempf and Ness \cite{kn1979} and its subsequent extension to real reductive algebraic groups by Richardson and Slodowsky \cite{rs1990}. The idea of interest is that, for suitable representations of such groups, one can study group orbits by understanding the infimum over the orbit of a certain vector space norm. For the purposes of this paper, it suffices to consider only representations of $\SL(d,\R)$ or $\SL(m,\R) \times \SL(d,\R)$ on vector spaces of tensors. In this context, the associated minimum vectors can be understood as normal forms of tensors and the actual numerical value of the infimum carries meaningful and important quantitative information about these tensors (in contrast to the usual situation in GIT in which one cares only about whether the infimum is zero or nonzero and whether or not it is attained). 

To begin the construction, suppose that $\mathcal A$ is any $k$-linear functional on a real vector space $V$ of dimension $d$. Appropriating the Kempf-Ness minimum vector calculations of GIT, it becomes possible to canonically associate a density functional $\aff : V^d \rightarrow \R_{\geq 0}$ to any such $\mathcal A$. Specifically, for any such $\mathcal A$ and any vectors $v_1,\ldots,v_d$, let $\aff(v_1,\ldots,v_d)$ be the quantity given by
\begin{equation}
 \begin{split}
 \aff(v_1,& \ldots,v_d) :=  \\
 & \left[ \inf_{M \in \mathrm{SL}(d,\R)} \sum_{j_1,\ldots,j_k=1}^d \left| \sum_{i_1,\ldots,i_k = 1}^d {\mathcal A} (M_{j_1 i_1} v_{i_1},\ldots, M_{j_k i_k} v_{i_k}) \right|^{2} \right]^{\frac{d}{2k}}.
 \end{split} \label{affdef}
 \end{equation}
 Before showing that the quantity \eqref{affdef} is a density functional, it is worthwhile to acknowledge the algebraic structure that lies behind it. When the $d$-tuple of vectors $v := (v_1,\ldots,v_d) \in V^{d}$ are linearly independent, one may define a representation $\rho^v_{\cdot} : \SL(d,\R) \times V \rightarrow V$ by setting
 \begin{equation} \rho^v_M(v_j) := \sum_{i=1}^d M_{ij} v_i \label{reptn} \end{equation}
 for each $j=1,\ldots,d$ and then extending to all of $V$ by linearity.
 This representation extends to act on $k$-linear functionals by duality, i.e., 
 \[ \left( \rho_{M}^v \mathcal A \right) (v_{j_1},\ldots,v_{j_k}) := \mathcal A (\rho^v_{M^T} v_{j_1},\ldots,\rho^v_{M^T} v_{j_k}), \] where $M^T$ is the transpose of $M$. If we further define a norm on the space of $k$-linear functionals by means of the formula
 \begin{equation*} || \mathcal A ||_v^2 := \sum_{j_1=1}^d \cdots \sum_{j_k = 1}^d |\mathcal A(v_{j_1},\ldots,v_{j_k})|^2  \end{equation*}
 then the formula \eqref{affdef} becomes
 \[  \aff(v_1,\ldots,v_d) = \left( \inf_{M \in \SL(d,\R)} || \rho_M^v \mathcal A||_v\right)^{\frac{d}{k}}. \]
 
 To see that $\aff$ is a density as promised, the first step is to demonstrate that $\aff(v_1,\ldots,v_d) = 0$ when $v_1,\ldots,v_d$ are linearly dependent. In this case, there must exist an invertible matrix $M$ such that $\sum_{i=1}^d M_{1 i} v_i = 0$, and without loss of generality, one may assume that this matrix $M$ has been normalized so as to belong to $\mathrm{SL}(d,\R)$. Now for each $t > 0$, let $M^{(t)}$ be the transpose of the matrix obtained by scalar multiplying the first row of $M$ by $t^{d-1}$ and all remaining rows by $t^{-1}$. These matrices $M^{(t)}$ belong to $\mathrm{SL}(d,\R)$ for all $t > 0$, and
 \[ \rho^v_{M^{(t)}} {\mathcal A} = t^{-k} \rho^{v}_{M} {\mathcal A} \]
by multilinearlity of $\mathcal A$ because 
\begin{equation} {\mathcal A}(M_{j_1 i_1}^{(t)} v_{i_1}, \ldots, M_{j_k i_k}^{(t)} v_{i_k}) = t^{-k} {\mathcal A} (M_{j_1 i_1} v_{i_1},\ldots, M_{j_k i_k} v_{i_k}) \label{scaling} \end{equation}
by homogeneity if each $j_1,\ldots,j_k$ is not equal to one, and if any index $j_\ell$ does equal one, then both sides vanish, making the equality \eqref{scaling} true trivially. Taking $t \rightarrow \infty$ shows that the infimum in \eqref{affdef} over all $\mathrm{SL}(d,\R)$ must vanish when $v_1,\ldots,v_d$ are linearly dependent. 

Now let $T$ be any linear transformation of $V$. When $v_1,\ldots,v_d$ are linearly dependent or when $T$ is not invertible, $T v_1,\ldots,T v_d$ will be linearly dependent, so it must hold that
\[ \aff(T v_1,\ldots, T v_d) = |\det T| \aff (v_1,\ldots,v_d) = 0. \]
Otherwise, when $v_1,\ldots,v_d$ are linearly independent and $T$ is invertible, there is a matrix $P \in \mathrm{GL}(d,\R)$ with $\det P = \det T$ such that $T v_j = \sum_{i=1}^d P_{ji} v_i$ for each $j=1,\ldots,d$. Factor $P$ as $\pm |\det P|^{\frac{1}{d}} P'$ for some $P'$ with $|\det P'| = 1$ in general and $\det P' =1$ when $d$ is odd. Once again, by multilinearity of $\mathcal A$, 
 \begin{align}
& \sum_{j_1,\ldots,j_k=1}^d  \left| \sum_{i_1,\ldots,i_k = 1}^d {\mathcal A} (M_{j_1 i_1} T v_{i_1},\ldots, M_{j_k i_k} T v_{i_k}) \right|^2 \nonumber
  \\  & \ \ =  | \det T|^{\frac{2k}{d}} \sum_{j_1,\ldots,j_k=1}^d \left| \sum_{i_1,\ldots,i_k,\ell_1,\ldots,\ell_k = 1}^d  {\mathcal A} (M_{j_1 i_1} P_{i_1 \ell_1} ' v_{\ell_1},\ldots,M_{j_k i_k} P_{i_k \ell_k}' v_{\ell_k}) \right|^2  \nonumber \\
  & \ \ =  |\det T|^{\frac{2k}{d}} \sum_{j_1,\ldots,j_k=1}^d  \left| \sum_{i_1,\ldots,i_k = 1}^d {\mathcal A} ((MP')_{j_1 i_1} v_{i_1},\ldots, (MP')_{j_k i_k}  v_{i_k}) \right|^2. \label{min0}
 \end{align} 
Since $\mathrm{SL}(d,\R)$ is a group, the set of matrices of the form $MP'$ when $M \in \SL(d,\R)$ is itself exactly $\SL(d,\R)$ assuming that $\det P' = 1$. If $\det P' = -1$, then the matrices $M P'$ for $M \in \SL(d,\R)$ are exactly those matrices $N$ which belong to $\SL(d,\R)$ after the first two rows of $N$ are interchanged. Since \eqref{min0} is invariant under permutations of the rows of $MP'$, it follows in both cases ($\det P' = \pm 1$) that
\begin{align*}
 \inf_{M \in \SL(d,R)}   \sum_{j_1,\ldots,j_k=1}^d & \left| \sum_{i_1,\ldots,i_k = 1}^d {\mathcal A} ((MP')_{j_1 i_1} v_{i_1},\ldots, (MP')_{j_k i_k}  v_{i_k}) \right|^2 \\ & = \left[ \aff(v_1,\ldots,v_d) \right]^{\frac{2k}{d}}, 
 \end{align*}
which gives the desired identity
\[ \mu_{\mathcal A}( T v_1,\ldots,T v_d ) = |\det T| ~ \mu_{\mathcal A}(v_1,\ldots,v_d) \]
for any $v_1,\ldots,v_d$ and any linear transformation $T$, as asserted.

\begin{example}
It is illuminating to compute $\aff$ in the specal case when $\mathcal A$ is a symmetric bilinear form. Fix linearly independent vectors $v_1,\ldots,v_d$ and define the matrix $A$ by $A_{i j} := {\mathcal A}(v_i,v_j)$. It follows that
\[ (\rho^v_M \mathcal A)(v_{j_1},v_{j_2}) = (M A M^T)_{j_1 j_2} \mbox{ and } || \rho^v_M \mathcal A||_v^2 = \mathrm{tr} ( M A M^T M A M^T). \]
Now $M A M^T M A M^T$ is symmetric and positive semidefinite, so its eigenvalues are all nonnegative. Thus the AM-GM inequality implies that
\[ d (\det ( M A M^T M A M^T))^\frac{1}{d} \leq \mathrm{tr} (M A M^T M A M^T)  \]
with equality when all eigenvalues are equal (which, when $A$ is invertible, can be attained for some $M \in \mathrm{SL}(d,\R)$ by building $M$ from a basis of unit-length eigenvectors of $A$ with respect to some inner product and then rescaling the eigenvectors appropriately). Because $\det M = 1$, $\det (M A M^T M A M^T) = (\det A)^2$ and therefore
\[ \aff (v_1,\ldots,v_d) = d^\frac{d}{4} |\det A|^{\frac{1}{2}}. \]
In particular, on a Riemannian manifold, setting $\mathcal A$ equal to the metric tensor $g$ yields a tensor density $\aff$ which is exactly equal to a dimensional constant times the corresponding Riemannian volume density. 
\end{example}

At this point, the reader may be somewhat understandably disappointed by the abstract nature of the infimum appearing in \eqref{affdef} since it is not immediately apparent how to compute the infimum in finitely many operations. However, the abstract nature of the definition \eqref{affdef} turns out to be a blessing rather than a curse, because it effectively allows the analysis to entirely sidestep the very deep and rich algebraic question of what \eqref{affdef} computes. It turns out that \eqref{affdef} is deeply connected, both algebraically and analytically, to the problem of finding polynomials in the entries of the tensor $A_{j_1,\ldots,j_k} = \mathcal A(v_{j_1},\ldots,v_{j_k})$ which are invariant under the action of the representation $\rho^v_\cdot$. Hilbert \cite{hilbert1893} showed that, when the group $\SL(d,\R)$ is replaced by $\SL(d,\C)$, there are finitely many polynomials invariant under the action of $\SL(d,\C)$ which generate the algebra of all such invariant polynomials. From this fact it is easy to see that the same result must be true for $\SL(d,\R)$ itself. A vast body of literature shows (via Weyl's unitarian trick) shows that the same result holds for representations of any group $G$ which is a real reductive algebraic group (which, as far as the present paper is concerned, is a class which includes $\SL(d,\R)$ and is closed under Cartesian products). It is possible in principle to compute these polynomials explicitly in finite time (see Sturmfels \cite{sturmfelsbook}), but in general going about the calculation in this way is somewhat unwieldy and akin to the computation of the determinant via its permutation expansion rather than by more efficient, symmetry-exploiting techniques. In any case, the density \eqref{affdef} simultaneously captures the behavior of all invariant polynomials at once, as demonstrated by the following lemma:
\begin{lemma}
Suppose that $G$ is a real reductive algebraic group and that $\rho$ is a $G$-representation on some finite-dimensional real vector space $V$ equipped with a norm $|| \cdot ||$. Let $p_1,\ldots,p_N$ be any collection of homogeneous polynomials of positive degree in $V$ which generate the algebra of all $G$-invariant polynomials. Then there exist constants $0 < C_1 \leq C_2 < \infty$ such that
\begin{equation} C_1 \max_{j=1,\ldots,N} |p_j(\mathcal A)|^{\frac{1}{d_j}} \leq \inf_{M \in G} ||\rho_M \mathcal A|| \leq C_2 \max_{j=1,\ldots,N} |p_j(\mathcal A)|^{\frac{1}{d_j}} \mbox{ for all } \mathcal A \in V. \label{norm} \end{equation}
\end{lemma}
\begin{proof}
To prove the first inequality, observe by scaling that
\[ ||p_j||^{-\frac{1}{d_j}}_\infty |p_j(\mathcal A)|^{\frac{1}{d_j}} \leq ||\mathcal A|| \]
for all $j$ and all $\mathcal A \in V$, where $||p_j||_\infty$ is the supremum of $p_j$ on the unit sphere of $||\cdot||$. Moreover, because each $p_j$ is invariant under $\rho$,
\[ ||p_j||^{-\frac{1}{d_j}}_\infty |p_j(\mathcal A)|^{\frac{1}{d_j}} = ||p_j||^{-\frac{1}{d_j}}_\infty |p_j(\rho_M \mathcal A)|^{\frac{1}{d_j}} \leq ||\rho_M \mathcal A||, \]
so taking an infimum in $M$ and a supremum in $j$ gives
\[ \left[ \min_{j=1,\ldots,N} ||p_j||_\infty^{-\frac{1}{d_j}} \right] \max_{j = 1,\ldots,N} |p_j(\mathcal A)|^{\frac{1}{d_j}} \leq \inf_{M \in G} ||\rho_M \mathcal A||. \]
To prove the reverse inequality, suppose for the sake of contradiction that it does not hold for any finite $C$. 
Because the inequality is homogeneous in the norm $|| \cdot ||$, its failure would imply that one could find a sequence $\mathcal A_k$ with $\inf_M || \rho_M \mathcal A_k|| = 1$ for all $k$ such that
\[ \max_{j} |p_j (\mathcal A_k)|^{\frac{1}{d_j}} \leq k^{-1} \inf_{M \in G} || \rho_M \mathcal A_k|| = k^{-1}. \]
Moreover, by replacing $\mathcal A_k$ by $\rho_{M_k} \mathcal A_k$ for suitable $M_k$ and taking a subsequence by compactness, it may be assumed that $\mathcal A_k$ converges to some $\mathcal A$ in the unit sphere as $k \rightarrow \infty$.  By continuity of the polynomials $p_j$, $p_j(\mathcal A) = 0$ for all $j$. Therefore $\mathcal A$ belongs to the so-called nullcone of the representation and by the real Hilbert-Mumford criterion, first proved by Birkes \cite{birkes1971}, there must exist a one-parameter subgroup $\rho_{\exp(t X)}$ of $G$ such that $\rho_{\exp(t X)}\mathcal A \rightarrow 0$ as $t \rightarrow \infty$. This, of course, implies that $\inf_M ||\rho_M \mathcal A|| = 0$. However $\inf_{M} ||\rho_M \mathcal A_k|| = 1$ for all $k$
implies that $||\rho_M \mathcal A_k|| \geq 1$ for all $M \in G$ and all $k$, which means by continuity that $||\rho_M \mathcal A|| \geq 1$ for all $M$, so $\inf_M ||\rho_M \mathcal A|| = 0$ must be contradicted.
\end{proof}
The inequality \eqref{norm} shows that the numerical value of $\aff(v_1,\ldots,v_d)$ is, in rough analogy with the symmetric bilinear form example just examined, comparable to the maximum of appropriate powers of the invariant polynomials applied to $A_{j_1,\ldots,j_k} = \mathcal A(v_{j_1},\ldots,v_{j_k})$. It is also worth observing that when many invariant polynomials exist (which, unlike the symmetric bilinear form case, is generally the more common situation), the nullcone of tensors $\mathcal A$ for which $\aff = 0$ will have codimension greater than one. In terms of affine curvature, this will mean that for general submanifolds of dimension $d$ in $\R^n$, it is typically ``easier'' to have nonvanishing affine curvature than it is in the case of hypersurfaces because the space of ``flat'' Taylor polynomial jets which must be avoided is often of codimension greater than one.

\subsection{Construction of the affine curvature tensor and associated measure}

\label{construction}
\psset{unit=15pt}
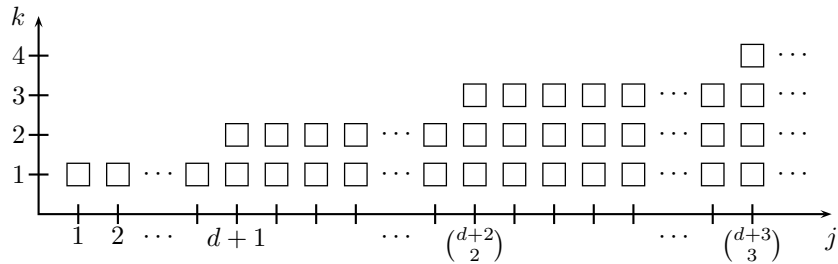
\begin{figure} 
\begin{center} 
\begin{pspicture}(0,0)(20,5)
\psline[arrows=<->](0,5)(0,0)(20,0)
\uput[-90](20,0){$j$} 
\uput[180](0,5){$k$}
\psline(-0.25,1)(0.25,1) \uput[180](0,1){$1$}
\psline(-0.25,2)(0.25,2) \uput[180](0,2){$2$}
\psline(-0.25,3)(0.25,3) \uput[180](0,3){$3$}
\psline(-0.25,4)(0.25,4) \uput[180](0,4){$4$}

\psline(1,-0.25)(1,0.25) \uput[-90](1,0){$1$}
\psline(2,-0.25)(2,0.25) \uput[-90](2,0){$2$}
\uput[-90](3.07,-0.1){$\cdots$} \uput[-90](3.07,1.5){$\cdots$}
\psline(4,-0.25)(4,0.25) \uput[-90](5,0){$d+1$}
\psline(5,-0.25)(5,0.25)
\psline(6,-0.25)(6,0.25) 
\psline(7,-0.25)(7,0.25)
\psline(8,-0.25)(8,0.25) 
\uput[-90](9.07,-0.1){$\cdots$} \uput[-90](9.07,1.5){$\cdots$} \uput[-90](9.07,2.5){$\cdots$}
\psline(10,-0.25)(10,0.25) 
\psline(11,-0.25)(11,0.25) \uput[-90](11,0){$\binom{d+2}{2}$}
\uput[-90](16.07,-0.1){$\cdots$} \uput[-90](16.07,1.5){$\cdots$} \uput[-90](16.07,2.5){$\cdots$} \uput[-90](16.07,3.5){$\cdots$}
 \uput[-90](19.07,1.5){$\cdots$} \uput[-90](19.07,2.5){$\cdots$} \uput[-90](19.07,3.5){$\cdots$} \uput[-90](19.07,4.5){$\cdots$}
\psline(12,-0.25)(12,0.25) 
\psline(13,-0.25)(13,0.25) 
\psline(14,-0.25)(14,0.25)
\psline(15,-0.25)(15,0.25) 
\psline(17,-0.25)(17,0.25) 
\psline(18,-0.25)(18,0.25)  \uput[-90](18,0){$\binom{d+3}{3}$}
\psdot[dotstyle=square,dotsize=10pt](1,1)
\psdot[dotstyle=square,dotsize=10pt](2,1)
\psdot[dotstyle=square,dotsize=10pt](4,1)
\psdot[dotstyle=square,dotsize=10pt](5,1)
\psdot[dotstyle=square,dotsize=10pt](5,2)
\psdot[dotstyle=square,dotsize=10pt](6,1)
\psdot[dotstyle=square,dotsize=10pt](6,2)
\psdot[dotstyle=square,dotsize=10pt](7,1)
\psdot[dotstyle=square,dotsize=10pt](7,2)
\psdot[dotstyle=square,dotsize=10pt](8,1)
\psdot[dotstyle=square,dotsize=10pt](8,2)
\psdot[dotstyle=square,dotsize=10pt](10,1)
\psdot[dotstyle=square,dotsize=10pt](10,2)
\psdot[dotstyle=square,dotsize=10pt](11,1)
\psdot[dotstyle=square,dotsize=10pt](11,2)
\psdot[dotstyle=square,dotsize=10pt](11,3)
\psdot[dotstyle=square,dotsize=10pt](12,1)
\psdot[dotstyle=square,dotsize=10pt](12,2)
\psdot[dotstyle=square,dotsize=10pt](12,3)
\psdot[dotstyle=square,dotsize=10pt](13,1)
\psdot[dotstyle=square,dotsize=10pt](13,2)
\psdot[dotstyle=square,dotsize=10pt](13,3)
\psdot[dotstyle=square,dotsize=10pt](14,1)
\psdot[dotstyle=square,dotsize=10pt](14,2)
\psdot[dotstyle=square,dotsize=10pt](14,3)
\psdot[dotstyle=square,dotsize=10pt](15,1)
\psdot[dotstyle=square,dotsize=10pt](15,2)
\psdot[dotstyle=square,dotsize=10pt](15,3)

\psdot[dotstyle=square,dotsize=10pt](17,1)
\psdot[dotstyle=square,dotsize=10pt](17,2)
\psdot[dotstyle=square,dotsize=10pt](17,3)
\psdot[dotstyle=square,dotsize=10pt](18,1)
\psdot[dotstyle=square,dotsize=10pt](18,2)
\psdot[dotstyle=square,dotsize=10pt](18,3)
\psdot[dotstyle=square,dotsize=10pt](18,4)
\end{pspicture}
\end{center}
\caption{In the diagram above, squares represent points in $\Z \times \Z$. The index set $\Lambda_{d,n}$ is simply the union of the first $n$ columns, and the homogeneous dimension $Q$ is simply the cardinality of $\Lambda_{d,n}$. In terms of the tensor $\mathcal A_p$, the number of boxes in each column indicates how many derivatives are applied to $f$ in the corresponding factor of the wedge product (or equivalently, the corresponding column of the matrix). \label{secondfig}}
\end{figure}

We move now to the construction of a covariant tensor which captures the affine curvature we are interested in. This tensor will be given an associated density using the formula \eqref{affdef} which can be integrated to give a canonical measure on immersed submanifolds $\mathcal M \subset \R^n$.

Suppose that $\mathcal M$ is a manifold of dimension $d$ which is equipped with a smooth immersion $f : \mathcal M \rightarrow \R^n$. For convenience, let the values of $f$ be regarded as column vectors.  For any positive integer $j$, let $\kappa_j$ be the smallest integer such that the dimension of the space $P_d^{\kappa_j}$ of real polynomials of degree $\kappa_j$ in $d$ variables has dimension at least $j+1$, and let $\Lambda_{d,n}$ be the index set
\[ \Lambda_{d,n} := \set{(j,k) \in \Z \times \Z}{ 1 \leq j \leq n \mbox{ and } 1 \leq k \leq \kappa_j  }. \]
The index set $\Lambda_{d,n}$ is represented pictorially in Figure \ref{secondfig} as the first $n$ columns of boxes.
The cardinality of $\Lambda_{d,n}$ is exactly the homogeneous dimension $Q$ defined in the introduction. We are going to define a $Q$-linear covariant tensor $\mathcal A_p$ at each point $p \in \mathcal M$ which captures the affine geometry of the immersion $f$. We will denote the action of $\mathcal A_p$ on $Q$-tuples of vectors by either
\[ {\mathcal A}_p(X_1,\ldots,X_Q) \mbox{ or } {\mathcal A}_p ((X_\lambda)_{\lambda \in \Lambda_{d,n}}) \]
depending on which approach is most convenient at the moment (where we lexicographically order the elements of $\Lambda_{d,n}$ when such an order is not otherwise specified).

Now for any finite sequence of vector fields $X_\lambda$ indexed by $\lambda \in \Lambda_{d,n}$, let
\begin{equation}  
\begin{split}
 {\mathcal A}_p( (X_\lambda)_{\lambda \in \Lambda_{d,n}}) := & \det (  X_{(1,1)} f(p) \wedge  \cdots \wedge   \\
 & X_{(j,1)} \cdots X_{(j,\kappa_j)} f(p) \wedge \cdots \wedge X_{(n,1)} \cdots X_{(n,\kappa_n)} f(p)).
 \end{split}  \label{adef}
 \end{equation}
Here the determinant of an $n$-fold wedge of vectors in $\R^n$ is understood to equal the determinant of the $n \times n$ matrix whose columns are the factors of the wedge (technically these factors are not unique, but the antisymmetry of the determinant and of wedge products guarantees the same determinant for any factorization). In other words, \eqref{adef} equals the determinant of an $n \times n$ matrix whose $j$-th column is the column vector $X_{(j,1)} \cdots X_{(j,\kappa_j)} f(p)$.
 (Note also that the lexicographic order on $\Lambda_{d,n}$ corresponds exactly to the order that each $\lambda \in \Lambda_{d,n}$ appears in the above formula when moving from left to right; with respect to Figure \ref{secondfig}, the order is left-to-right followed by bottom-to-top.)
 
This object $\mathcal A_p$ will be called the affine curvature tensor at $p$. First observe that it is certainly linear in $X_\lambda$ for each $\lambda \in \Lambda_{d,n}$. To see that $\mathcal A_p$ depends only on the pointwise values of the $X_\lambda$ at $p$ and not any derivatives of these vector fields, it suffices to show that any single one of the vector fields $X_\lambda$ may be replaced by any other vector field $X_\lambda'$ agreeing with $X_\lambda$ at $p$ without changing the value of $\mathcal A_p$. For any indices $\lambda = (j,k)$ such that $\kappa_j = 1$, this invariance under replacement follows immediately from the fact that these vector fields appear alone in their own column (i.e., the formula \eqref{adef} contains no derivatives of $X_\lambda$ to begin with). For any $\lambda = (j,k)$ with $\kappa_j > 1$, the identity
\begin{align*}
X_{(j,1)} & \cdots  X_{(j,k)} \cdots X_{(j,\kappa_j)} f(p) - X_{(j,1)} \cdots X'_{(j,k)} \cdots X_{(j,\kappa_j)} f(p) \\
& = X_{(j,1)} \cdots X_{(j,k-1)} [ X_{(j,k-1)}, X_{(j,k)} - X'_{(j,k)}] X_{(j,k+1)} \cdots X_{(j,\kappa_j)} f(p) \\
& + \cdots + [X_{(j,1)}, X_{(j,k)} - X_{(j,k)}'] X_{(j,2)} \cdots \widehat{X_{(j,k)}} \cdots X_{(j,\kappa_j)} f(p)
\end{align*}
(where $\widehat{\cdot}$ indicates omission of $X_{(j,k)}$ in its usual place)  shows that $\mathcal A_p$ vanishes when $X_{\lambda}$ is replaced by $X_{\lambda} - X'_{\lambda}$: the number of columns of the matrix in \eqref{adef} for which $f$ is differentiated to some order between $1$ and $\kappa_j - 1$ is strictly greater than the dimension of the vector space generated by such operators, so there must be linearly dependent columns in the matrix, which forces $\mathcal A_p$ to vanish.

The measure $\aff$ on the submanifold which will be shown under suitable additional hypotheses to satisfy \eqref{oberlin} is exactly that measure whose density is given from the tensor $\mathcal A_p$ by the formula \eqref{affdef}.

\subsection{Proof of Parts \ref{ocfails} and \ref{ocub} of Theorem \ref{mainthm}}

We begin with the following elementary lemma which gives an estimate for the volume of the convex hull of certain sets $S \subset \R^n$:
\begin{lemma}
Suppose $S \subset \R^n$ is a compact set containing the origin, and let $K$ be its convex hull. There exist $v_1,\ldots,v_n \in S$ such that the sets \label{convexlem}
\[ K_1 := \set{v \in \R^n}{v = \sum_{i=1}^n c_i v_i \mbox{ for coefficients } c_i \mbox{ such that } \sum_{i=1}^n |c_i| \leq 1} \]
and 
\[ K_\infty := \set{v \in \R^n}{v = \sum_{i=1}^n c_i v_i \mbox{ for coefficients } c_i \in [-1,1], \ i = 1,\ldots,n} \]
satisfy
\begin{equation} K_1 \subset K \subset K_\infty. \label{containment} \end{equation}
In particular,
\begin{equation} \frac{2^n}{n!} |\det (v_1 \wedge \cdots \wedge v_n)| \leq |K| \leq 2^n |\det (v_1 \wedge \cdots \wedge v_n)|. \label{volume} \end{equation}
\end{lemma}
\begin{proof}
Let $V$ be the unique vector subspace of $\R^n$ of smallest dimension which contains $S$ (where uniqueness holds because the intersection of two subspaces containing $S$ would be a subspace of smaller dimension also containing $S$). Let $m$ denote the dimension of $V$, and let $\det_V$ be any nontrivial alternating $m$-linear form on $V$.
  Let $(v_1,\ldots,v_m) \in S^m$ be any $m$-tuple at which the maximum of the function
\[ (v_1,\ldots,v_m) \rightarrow |\det_V (v_1 \wedge \cdots \wedge v_m)| \]
is attained. Since $S$ is not contained in any subspace of smaller dimension, $|\det_V (v_1 \wedge \cdots \wedge v_m)| > 0$ unless $m=0$ (in which case $S = \{0\}$ and the lemma is trivial). Now by Cramer's rule, for any $v \in V$, 
\[ v = \sum_{i=1}^m (-1)^{i-1}\frac{\det_V ( v \wedge v_1 \wedge \cdots \wedge  \widehat{v_i} \wedge \cdots \wedge v_m)}{\det_V (v_1 \wedge \cdots \wedge v_m)} v_i, \]
where, in this case, the circumflex $\widehat{\cdot}$ indicates that a vector is to be omitted from the determinant. In the particular case when $v \in S$, the $m$-tuple $(v,v_1,\ldots,\widehat{v_i},\ldots,v_m)$ belongs to the set $S^m$ over which the supremum of $|\det_V|$ was taken; therefore each numerator has magnitude less than or equal to the denominator. Thus $S$ belongs to the parallelepiped 
\[ P := \set{v \in \R^n}{v = \sum_{i=1}^m c_i v_i \mbox{ for some } c_1,\ldots,c_m \in [-1,1]}. \]
Since $P$ is convex and contains $S$, it must contain $K$ as well.  To establish the lemma, we extend the sequence $v_1,\ldots,v_m$ to a sequence of length $n$ by fixing $v_j = 0$ for $j > m$. Trivially $P = K_\infty$ for this choice, so the containment $K \subset K_\infty$ must hold. For the remaining containment, observe that $v_1,\ldots,v_n$ must belong to $K$ since they belong to $S$. Therefore, by convexity of $K$, the set $K_1$ must be contained in $K$. The volume inequality \eqref{volume} follows from the elementary calculation of the volumes of $K_1$ and $K_\infty$.
\end{proof}

With Lemma \ref{convexlem} in place, we turn now to the proof of Parts \ref{ocfails} and \ref{ocub} of Theorem \ref{mainthm}.
Pick any point $p \in {\mathcal M}$ and fix any smooth coordinate system $(t_1,\ldots,t_d)$ near $p$ so that the immersion $f : \mathcal M \rightarrow \R^d$ may be regarded in these coordinates as a function from a $3 \delta$ neighborhood of the origin (chosen so that $t=0$ are the coordinates of $p$) into $\R^n$. By Taylor's formula, for all $t_0,t \in \R^d$ with $|t_0| \leq \delta, |t| \leq 2\delta$, 
\begin{equation} f(t)-f(t_0) = \sum_{0< |\alpha| \leq \ell} \frac{(t-t_0)^\alpha }{\alpha!} \partial_t^\alpha f(t_0) + \sum_{|\beta|=\ell+1} \frac{(t-t_0)^\beta}{\beta!} R^\beta_{t_0}(t),\label{ts} \end{equation}
for any finite $\ell$, where each remainder term $R^\beta_{t_0}(t)$ is continuous on $|t| \leq 2 \delta$ and equals $\partial_t^\beta f(t_0)$ when $t = t_0$. (For most of what follows, $t_0$ will be regarded as a fixed but otherwise arbitrary point with $|t_0| \leq \delta$.) For definiteness, let $\ell := \kappa_n$, i.e., $\ell$ equals the highest order of differentiation that appears in a column of the matrix whose determinant forms $\mathcal A$ (or equivalently, $\ell$ is the number of boxes in column $n$ of the diagram given in Figure \ref{secondfig}). This choice of $\ell$ implies that the dimension of the space of polynomials of degree $\ell$ with no constant term is at least equal to $n$. For any $r \in (0,\delta]$, let $S_{t_0,r}$ be the compact subset of $\R^n$ given by 
\[ S_{t_0,r} := \{ 0 \} \cup \bigcup_{|\alpha| \leq \ell} \left\{ \frac{r^{|\alpha|}}{\alpha!} \partial_t^\alpha f(t_0) \right\} \cup \bigcup_{|\beta| = \ell+1, |t| \leq 2\delta} \left\{ \frac{r^{\ell+1}}{\beta!} R^\beta_{t_0} (t) \right\} \]
and let $K_{t_0,r}$ be the convex hull of $S_{t_0,r} \cup (-S_{t_0,r})$.
Now each term in either sum on the right-hand side of \eqref{ts} belongs to $K_{t_0,r}$ whenever $|t| \leq 2 \delta$ and $|t-t_0| \leq r$. Because the total number of summands on the right-hand side is at most some constant $C$ depending only on $d$ and $n$, the difference vector $f(t) - f(t_0)$ must belong to the dilated set $C K_{t_0,r}$ whenever $|t| \leq 2 \delta$ and $|t-t_0| \leq r$. In particular, this implies that the translated set $C K_{t_0,r} + f(t_0)$ must contain the vector $f(t)$ whenever $|t| \leq 2\delta$ and $|t-t_0| \leq r \leq \delta$.

By virtue of \eqref{volume}, the Lebesgue measure of the set $C K_{t_0,r} + f(t_0)$ is  $O(r^Q)$ as $r \rightarrow 0^+$ since it is dominated by a constant depending on $d$ and $n$ times a determinant $|\det (v_1 \wedge \cdots \wedge v_n)|$ for some $v_1,\ldots,v_n \in S_{t_0,r} \cup (-S_{t_0,r})$ and since $Q$ is by definition the smallest integer which it is possible to express as a sum of degrees of distinct, nonconstant monomials in $d$ variables (thus $Q$ corresponds to the smallest possible factor of $r$ which will appear via scaling in such determinants). In fact, a slightly stronger result is also true: namely, that it is possible to quantify the implied constant in this $O(r^Q)$ estimate in terms of the affine curvature tensor $\mathcal A$ at $t_0$.  For any collection $\alpha_1,\ldots,\alpha_n$ of monomials such that $|\alpha_1| + \cdots + |\alpha_n| = Q$, it is possible to find indices $i_\lambda$ for each $\lambda \in \Lambda_{d,n}$ (these indices being obtained by ``expanding'' each $\alpha_i$ as a composition of first-order coordinate derivatives) so that
\[ |\det (\partial^{\alpha_1}_t f(t_0) \wedge \cdots \wedge \partial^{\alpha_n}_t f(t_0))| = | {\mathcal A}_{t_0} ((\partial_{t_{i_\lambda}})_{\lambda \in \Lambda_{d,n}})|. \]
Therefore it follows from \eqref{volume} that when $r \leq \delta$, the image $f(B_r(t_0))$ is contained in $C K_{t_0,r} + f(t_0)$, which is a compact convex set with volume no greater than
\[ C' r^Q \left[ \sum_{j_1,\ldots,j_Q=1}^d \left| {\mathcal A}_{t_0} (\partial_{t_{j_1}},\ldots,\partial_{t_{j_Q}}) \right|^2 \right]^{\frac{1}{2}} + O(r^{Q+1}) \]
as $r \rightarrow 0^+$, where $C'$ is some new constant depending only on $d$ and $n$.
Consequently, if $\mu$ is any measure on $\mathcal M$ satisfying the restricted Oberlin condition \eqref{oberlin} with exponent $\alpha$ and constant $C_\mu$, then
\begin{equation} \limsup_{r \rightarrow 0^+} r^{-\alpha Q} \mu(B_r(t_0)) \lesssim C_\mu \left| \left[ \sum_{j_1,\ldots,j_Q=1}^d \left| {\mathcal A}_{t_0} (\partial_{t_{j_1}},\ldots,\partial_{t_{j_Q}}) \right|^2 \right]^{\frac{1}{2}} \right|^{\alpha} \label{oberlineq} \end{equation}
for any $t_0$ with $|t_0| \leq \delta$ with an implied constant depending only on $d$ and $n$.
If $\alpha \geq d/Q$, this implies that $\mu$ must be absolutely continuous with respect to Lebesgue measure on $\mathcal M$ on a $\delta$-neighborhood of the chosen origin point $p$, and if $\alpha > d/Q$, it further implies that $\mu$ must be the zero measure on that neighborhood (since the Radon-Nykodym derivative of $\mu$ with respect to Lebesgue measure must vanish at every Lebesgue point, which is almost every point in the neighborhood), thus establishing Part \ref{ocfails} of Theorem \ref{mainthm}. When $\alpha = d/Q$, because $\mu$ must be absolutely continuous with respect to Lebesgue measure, it must follow that
\[ \limsup_{r \rightarrow 0^+} r^{-d} \mu(B_r(t_0)) = c_d \frac{d \mu}{dt} (t_0) \]
for almost every $t_0$ with $|t_0| \leq \delta$, where $d \mu / dt$ is the Radon-Nykodym derivative of $\mu$ with respect to Lebesgue measure $dt$ in the chosen coordinate system.
By \eqref{oberlineq}, then, 
\[ \left. \frac{d \mu}{dt} \right|_q \lesssim C_\mu \left[ \sum_{j_1,\ldots,j_Q=1}^d \left| {\mathcal A}_{q} (\partial_{t_{j_1}},\ldots,\partial_{t_{j_Q}}) \right|^2 \right]^{\frac{d}{2Q}} \]
for almost every point $q$ in some neighborhood of the original point $p$ (where, once again, the implied constant depends only on $d$ and $n$). Now, by transforming the coordinates $(t_1,\ldots,t_d)$ by matrices $M \in \mathrm{SL}(d,\R)$ to produce new coordinate systems, it follows by the same reasoning as above that 
\begin{align*}
\left. \frac{d \mu}{dt} \right|_q \lesssim C_\mu \left[ \sum_{j_1,\ldots,j_Q=1}^d \left| \sum_{i_1,\ldots,i_Q=1}^d {\mathcal A}_{q} (M_{j_1 i_1} \partial_{t_{i_1}},\ldots,M_{j_Q i_Q} \partial_{t_{i_Q}}) \right|^2 \right]^{\frac{d}{2Q}} 
  \end{align*}
for every $M \in \mathrm{SL}(d,\R)$ and almost every $q$ in a neighborhood of $p$ (by continuity of the right-hand side as a function of $M$, it suffices to consider only some countable dense subset of $\mathrm{SL}(d,\R)$ so that the set on which the inequality fails is clearly null). Taking an infimum over $M$ gives that
\[ \frac{d \mu}{dt} \lesssim C_\mu \frac{d \aff}{dt} \]
almost everywhere on the coordinate patch. Because the coordinates and patch were arbitrary, it follows that Part \ref{ocub} of Theorem \ref{mainthm} must hold with an implicit constant which equals a dimensional quantity (depending only on $d$ and $n$) times the  Oberlin constant $C_\mu$ from \eqref{oberlin} for the measure $\mu$ itself.

\section{Sufficiency and nontriviality}
\label{suffsec}
\subsection{On the geometry of functions on measurable sets}

This section begins with a construction generalizing the results of Theorem 1 of \cite{gressman2009}. Roughly stated, that theorem indicated that for single-variable real polynomials of a given degree, every measurable subset of the real line has a ``core'' which contains a nontrivial fraction of the set such that the supremum of any such polynomial (or appropriately weighted derivatives) on the core is bounded above by the average of the polynomial on the entire set. The proof involved careful analysis of Vandermonde determinants and has no immediate generalization to other dimensions or families of functions. In the arguments below, an entirely different approach will be used which is based on convex geometry and admits extensions to a variety of new contexts. In particular, the setting of polynomials is no simpler to study than any other finite-dimensional family of real analytic functions, which will be the preferred formulation of the result.


To formulate the result, it is convenient to make the following definition.
Let $\M$ be any real analytic manifold of dimension $d$ and let $\F$ be a finite-dimensional vector space of real analytic functions on $\M$ whose differentials span the cotangent space at every point of $\M$. Any such pair $(\M,\F)$ will be called a geometric function system. Such a system will be called compact when either $\M$ is compact or has a compact closure in some larger real analytic manifold $\M^+$ such that the functions of $\F$ extend to functions $\F^+$ on $\M^+$ in such a way that $(\M^+,\F^+)$ is also a geometric function system. The first result is a ``zeroth order'' version of the results of \cite{gressman2009}:
\begin{lemma}
Suppose $(\M,\F)$ is a compact geometric function system. Then for any finite positive measure $\mu$ on $\M$ absolutely continuous with respect to Lebesgue measure and any measurable set $E \subset \M$ of positive measure, there is a \label{noderivlem} measurable subset $E' \subset E$ such that $\mu(E') \gtrsim \mu(E)$ and
\[ \sup_{p \in E'} |f(p)| \lesssim \frac{1}{\mu(E)} \int_E |f| ~ d \mu \mbox{ for all } f \in \F. \]
The implicit constants in both inequalities depend only on the pair $(\M,\F)$.
\end{lemma}
\begin{proof}
For each $f \in \F$, consider the following norm:
\[ ||f|| := \frac{1}{\mu(E)} \int_E |f| d \mu. \]
Compactness of the geometric function system implies that $||f||$ is finite for every $f \in {\mathcal F}$.
Because each function $f \in \F$ is real analytic and the measure of $E$ is strictly positive, no $f \in \F$ aside from the zero function can have $||f|| = 0$, which is what guarantees that $|| \cdot ||$ is a norm rather than merely a seminorm. Assuming that the dimension of $\mathcal F$ is $k$, applying Lemma \ref{convexlem} to the set $S$ which is the unit sphere of $|| \cdot ||$ and using homogeneity of the norm, there must be functions $f_1,\ldots,f_k$ with $||f_i||=1$ for all $i$ (none of the functions $f_i$ will be identically zero because the unit sphere does not lie in any nontrivial subspace of $\mathcal F$) such that every $f \in \F$ has the property that
\[ f = \sum_{i=1}^k c_i f_i \]
with $|c_i| \leq ||f||$ for each $i$.  In particular, this implies that
\[ |f(p)| = \left| \sum_{i=1}^k c_i f_i(p) \right| \leq ||f|| \sum_{i=1}^k |f_i(p)| \]
for each $f \in {\mathcal F}$.
Let $E'$ be the subset of $E$ on which $\sum_{i=1}^k |f_i(p)| \leq 2k$; by Tchebyshev's inequality,
 $$\mu(E') \geq \mu(E) - \frac{1}{2k} \int_E \sum_{i=1}^k |f_i(p)| d \mu \geq \frac{1}{2} \mu(E)$$
 and
 \[ \sup_{p \in E'} |f(p)| \leq \sup_{p \in E'} \left[ ||f|| \sum_{i=1}^k |f_i(p)| \right] \leq \frac{2k}{\mu(E)} \int_E |f| d \mu \]
 for all $f \in \F$.
\end{proof}

The extension of the results of \cite{gressman2009} to derivative estimates in higher dimensions is necessarily much more subtle than the one-dimensional case because of inherent issues of anisotropy of differentiation in differing directions. Any proper formulation will necessarily be phrased in terms of vector fields which capture (either implicitly or explicitly) this anisotropy. The formulation to be used here is as follows:
\begin{lemma}
Suppose $(\M,\F)$ is a compact geometric function system and let $N$ be any positive integer. Then for any  finite positive measure $\mu$ on $\M$ absolutely continuous with respect to Lebesgue measure and any measurable set $E \subset \M$ of positive measure, there is an open set $U \subset \M$, a family of smooth vector fields $\{X_{j,i}\}_{j,i}$ with $j \in \{1,\ldots,N\}$ and $i \in \{1,\ldots,d\}$ and a measurable set $E' \subset E \cap U$ such that the following are true with implicit constants depending only on the pair $(\M,\F)$ and the integer $N$: \label{derivlemm}
\begin{itemize}
\item The subset $E' \subset E \cap U$ satisfies $\mu(E') \gtrsim \mu(E)$.
\item The vector fields $X_{j,i}$ satisfy $\inf_{p \in E'} \at{ \mu( X_{j,1} \wedge \cdots \wedge X_{j,d})}_p \gtrsim \mu(E)$
and
\begin{equation} X_{j,i} = \sum_{i' = 1}^d c_{j,i,i'} X_{j-1,i'} \label{changeofbasis} \end{equation}
with $|c_{j,i,i'}| \lesssim 1$ for each $j \in \{2,\ldots,N\}$ and each $i,i' \in \{1,\ldots,d\}$. Here $\mu(X_1 \wedge \cdots \wedge X_d)$ equals the volume of the parallelepiped generated by $X_1,\ldots,X_d$ as measured by $\mu$, which more formally is defined to equal Radon-Nykodym derivative $d \mu / dt$ with respect to some coordinate system $(t_1,\ldots,t_d)$ times the absolute value of the determinant of the matrix ${\mathbf X}$ with entries ${\mathbf X}_{ij} = dt_i(X_j)$.
\item For any indices $i_1,\ldots,i_N \in \{1,\ldots,d\}$, 
\begin{equation} \sup_{ p \in E'} \left| X_{N,i_N} \cdots X_{1,i_1} f(p) \right| \lesssim \frac{1}{\mu(E)} \int_E |f| d \mu \label{derivest} \end{equation}
uniformly for all $f \in \F$.
\end{itemize}
\end{lemma}
\begin{proof}
By induction (the base case of which is taken to be Lemma \ref{noderivlem}), for a given measurable set $E \subset \M$ of positive measure, we may assume that there exist a nested family of open sets $\M =: U_0 \supset U_1 \supset U_2 \supset \cdots \supset U_{N-1}$ and vector fields $\{X_{j,i}\}$, $i=1,\ldots,d$, defined on $U_j$ for each $j \in \{1,\ldots,N-1\}$ satisfying all the stated properties. Next, let $\F_0 := \F$ and then take $\F_j$ to be the vector space of real analytic functions on $U_j$ spanned by $\F$ and all functions of the form $X_{j,i} f$ for $i = 1,\ldots,d$ and $f \in \F$. By construction, $\mu(E \cap U_{N-1}) \gtrsim \mu(E) > 0$, so in particular,
\[ ||f||_{N-1} :=  \frac{1}{\mu(E \cap U_{N-1})} \int_{E \cap U_{N-1}} |f| d \mu \]
will be a norm on $\F_{N-1}$. Let $f_1,\ldots,f_k$ be a basis of $\F_{N-1}$ given by applying Lemma \ref{convexlem} to the unit sphere of $||\cdot ||_{N-1}$, let $\mathcal I$ be the set of $d$-tuples $\beta := (\beta_1,\ldots,\beta_d)$ of indices satisfying $1 \leq \beta_1 < \beta_2 < \cdots < \beta_d \leq k$, and let $V_\beta$ be the open set
\[ \set{ p \in U_{N-1}}{ \left| \at{ d f_{\beta_1} \wedge \cdots \wedge df_{\beta_d}}_p \right| > \frac{1}{2} \left| \at{ d f_{\beta_1'} \wedge \cdots \wedge df_{\beta_d}'}_p \right| \ \forall \beta' \in {\mathcal I} \setminus \{\beta\}}. \]
Because the cardinality of $\mathcal I$ is bounded by a constant depending only on $N$ and the dimensions of $\F$ and $\M$, there is at least one $\beta$ such that $\mu(E \cap V_\beta) \gtrsim \mu(E \cap U_{N-1})$, where the implicit constant may simply be taken to be $(\# \Lambda)^{-1}$.  If we now define the vector field $X_{N,i}$ on the set $U_N := V_\beta$ to equal
\[ X_{N,i} f := \frac{ d f_{\beta_1} \wedge \cdots \wedge df_{\beta_{i-1}} \wedge df \wedge df_{\beta_{i+1}} \wedge \cdots \wedge df_{\beta_d}}{df_{\beta_1} \wedge \cdots \wedge df_{\beta_d}}, \]
then it must be the case that
\[ X_{N,i} f(p) = \sum_{i=1}^k c_i \frac{ \left. d f_{\beta_1} \wedge \cdots \wedge df_{\beta_{i-1}} \wedge df_i \wedge df_{\beta_{i+1}} \wedge \cdots \wedge df_{\beta_d} \right|_p }{ \left. df_{\beta_1} \wedge \cdots \wedge df_{\beta_d} \right|_p } \]
for constants $c_i$ satisfying $|c_i| \leq ||f||_{N-1}$. Since the ratio is bounded above by $2$ on $U_N$, it follows that
\begin{equation} \sup_{p \in U_N} |X_{N,i} f(p)| \leq \frac{2k}{\mu(E \cap U_{N-1})} \int_{E \cap U_{N-1}} |f| d \mu \label{derivest2} \end{equation}
for all $f \in \F_{N-1}$. Since $\mu(E \cap U_{N-1}) \gtrsim \mu(E)$, it follows by induction that
\[ \sup_{p \in U_N} |X_{N,i_N} \cdots X_{1,i_1} f(p)| \lesssim \frac{1}{\mu(E)} \int_{E} |f| d \mu \]
for all $f \in \F$. This establishes \eqref{derivest} for any set $E' \subset E \cap U_N$.

Next observe that each vector field $X_{N,i}$ is locally a coordinate vector field relative to the coordinate functions $(f_{\beta_1},\ldots,f_{\beta_d}) \in \F_{N-1}^d$ such that the average value of $|f_{\beta_i}|$ on $E \cap U_{N-1}$ is $1$. By induction, we may assume the corresponding fact is true for the vector fields $X_{N-1,i}$. In particular, if $(g_{\alpha_1},\ldots,g_{\alpha_n})$ are the coordinate functions for the vector fields $X_{N-1,i}$, then it follows that
\[ X_{N,i} = \sum_{i'=1}^d (X_{N,i} g_{\alpha_{i'}}) X_{N-1,i'}. \]
By the derivative estimate \eqref{derivest2}, assuming $N \geq 2$,
\begin{align*} 
\sup_{p \in U_N}  |X_{N,i} g_{\alpha_{i'}}(p)| & \leq \frac{2 \dim \F_{N-1}}{\mu(E \cap U_{N-1})} \int_{U_{N-1} \cap E} |g_{\alpha_{i'}}| d \mu \\
& \lesssim \frac{1}{\mu(E \cap U_{N-2})} \int_{E \cap U_{N-2}} |g_{\alpha_{i'}}| d \mu \lesssim 1,
\end{align*}
which is exactly the bound on the coefficients $c_{j,i,i'}$ claimed for \eqref{changeofbasis}.

Lastly, the quantity $\mu(X_{j,1} \wedge \cdots \wedge X_{j,d})$ must be estimated. Observe that $\mu(X_{j,1} \wedge \cdots \wedge X_{j,d})$ is exactly the Radon-Nykodym derivative of $\mu$ with respect to Lebesgue measure in coordinates given by $f_{\beta_1},\ldots,f_{\beta_d}$. This implies that
meaning that
\[ \int_{E'} \left[ \mu( X_{N,1} \wedge \cdots \wedge X_{N,d}) \right]^{-1} d \mu = \int_{E'} | d f_{\beta_1} \wedge \cdots \wedge df_{\beta_d}|. \]
Because the functions are real analytic, we know that there is a finite number $M$ independent of the choice of the functions $f_{\beta_i}$ such that the system $f_{\beta_i}(p) = c_i$ has at most $M$ nondegenerate solutions (at which the Jacobian is nonzero). For polynomial functions, this is a simple consequence of B\'{e}zout's Theorem. In our case, however, even if the original function system ${\mathcal F}_0$ consists only of polynomial functions, the definition of the $X_{j,i}$ lead naturally to the inclusion of certain rational functions in ${\mathcal F}_i$, at which point there is not much additional difficulty in going to the more general context of real analytic functions. The algebraic argument is given in Section \ref{appendix} and for now may be safely postponed.

Assuming the existence of such an $M$ depending only on the geometric function system, by the change of variables formula,
\[ \int_{E'} |d f_{\beta_1} \wedge \cdots \wedge df_{\beta_d}| \leq M \prod_{j=1}^d | f_{\beta_i}(E')|, \]
where $|f_{\beta_i}(E')|$ refers to the one-dimensional Lebesgue measure of the image of $E'$ via $f_{\beta_i}$.
Because the average value of $f_{\beta_i}$ on $E \cap U_{N-1}$ is $1$, by Lemma \ref{noderivlem}, there is a subset $E' \subset E \cap U_{N}$ with $\mu(E') \gtrsim \mu(E \cap U_{N}) \gtrsim \mu(E)$ such that
\[ \sup_{p \in E'} |f_{\beta_i}(p)| \lesssim \frac{1}{\mu(E \cap U_N)} \int_{E \cap U_N} |f_{\beta_i}| d \mu \lesssim \frac{1}{\mu(E \cap U_{N-1})} \int_{E \cap U_{N-1}} |f_{\beta_i}| d \mu \lesssim 1 \]
for each $i$, which implies that $|f_{\beta_i}(E')| \lesssim 1$ for each $i$ as well.
For this set $E'$, it follows that
\[ \int_{E'} \left[ \mu( X_{N,1} \wedge \cdots \wedge X_{N,d}) \right]^{-1} d \mu \lesssim 1. \]
Further restricting $E'$ using Tchebyshev's inequality, we may assume that 
\[ \inf_{E'} \mu(X_{N,1} \wedge \cdots \wedge X_{N,n}) \gtrsim \mu(E). \]
This completes the proof.
\end{proof}

\subsection{Proof of Part \ref{ocineq} of Theorem \ref{mainthm}}

We now return to the proof of Part \ref{ocineq} of Theorem \ref{mainthm}. The proof combines Lemma \ref{derivlemm} with the geometric framework introduced in  Section \ref{construction}. Suppose that $\mathcal M$ is a real analytic manifold of dimension $d$ and that $f$ is a real analytic immersion of ${\mathcal M}$ into $\R^n$ in such a way that the component functions $f_1,\ldots,f_n$ of the immersion together with the constant function belong to some compact geometric function system $(\M,\F)$. Fix any compact convex set $K \in \mathcal K_n$, let $E := f^{-1}(K)$, and let $p_0 \in E$. Now the integral
\[ I(E) := \frac{1}{\aff(E)^n} \int_{E^n} \! \! \! |\det( f(p_1) - f(p_0),\ldots,f(p_n) - f(p_0))| d \aff(p_1) \cdots d \aff(p_n) \]
must be bounded above by $n! |K|$ since the integrand equals $n!$ times the volume of the simplex generated by $f(p_0),\ldots,f(p_n)$, which has volume bounded by $|K|$ since each point belongs to $K$ and $K$ is convex.
In this case Lemma \ref{derivlemm} can be applied to each integral iteratively to prove a lower bound for the functional. Specifically, the lemma is applied to the innermost integral, which is then replaced by a supremum over some set $E'$ of some derivative in the parameter $p_1$. As a result, the lemma establishes that
\begin{align*}
 I(E) \gtrsim \sup_{(p_1,\ldots,p_n) \in (E')^n} | \det (& X_{1,i_{(1,1)}} f(p_1) \wedge  \cdots \wedge  \\
 & X_{\kappa_j,i_{(j,1)}} \cdots X_{1,i_{(j,\kappa_j)}} f(p_j) \wedge \cdots \wedge \\
 & X_{\kappa_n,i_{(n,1)}} \cdots X_{1,i_{(n,\kappa_n)}} f(p_n)) |
 \end{align*}
for any choice of indices $i_{\lambda}$ for $\lambda \in \Lambda_{d,n}$. Next replace the supremum over $(p_1,\ldots,p_n) \in E'^n$ by a supremum over $p \in E'$ assuming $p_1 = \cdots = p_n = p$.  It is also advantageous to use only vector fields $X_{\kappa_n,i'}$ rather than using any $X_{j,i}$ for $j < \kappa_n$. Thanks to \eqref{changeofbasis} it must be the case that
\begin{align*}
\det & \left(X_{\kappa_n,i_{(1,1)}'} f(p) \wedge \cdots \wedge X_{\kappa_n,i_{(n,1)}} \cdots X_{\kappa_n,i_{(n,\kappa_n)}'} f(p) \right) \\
 &  = \sum_{i} c_{ii'} \det \left(X_{1,i_{(1,1)}} f(p) \wedge \cdots \wedge X_{\kappa_n,i_{(n,1)}} \cdots X_{1,i_{(n,\kappa_n)}} f(p) \right) 
 \end{align*}
 with coefficients $|c_{ii'}| \lesssim 1$ where the sum is over all possible choices of the indices $i_\lambda$. This identity holds because the change of basis formula may be simply substituted term-by-term in the left-hand side of the equation; any terms in which the coefficients of the change of basis happened to be differentiated by some subsequent vector field would ultimately have determinant zero since (assuming the column in which the derivative appears is column $j$), the number of derivatives acting directly on $f$ would be strictly less than $\kappa_j$, which means that column $j$ and all preceding columns would be linearly dependent. Therefore by the triangle inequality, it must be the case that
\begin{align*}
 I(E) \gtrsim \sup_{p \in E'} | \det (& X_{\kappa_n,i_{(1,1)}'} f(p) \wedge  \cdots \wedge   
  X_{\kappa_n,i_{(n,1)}'} \cdots X_{\kappa_n,i_{(n,\kappa_n)}} f(p)) |
 \end{align*}
 uniformly for any choice of $i'_\lambda$. Taking an $\ell^2$ norm over all such choices and invoking the definition \eqref{affdef} of the density $\aff$
 \begin{align*}
 I(E) \gtrsim \sup_{p \in E'} \left[ \left. \aff(X_{\kappa_n,1},\ldots,X_{\kappa_n,d}) \right|_p \right]^\frac{Q}{d}.  
 \end{align*}
 To conclude, observe that for the measure $\aff$, the quantity $\aff(X_{\kappa_n,1},\ldots,X_{\kappa_n,d})$ exactly equals the geometric density $\aff(X_{\kappa_n,n} \wedge \cdots \wedge  X_{\kappa_n,n})$ bounded below by Lemma \ref{derivlemm}. Therefore
\[ |K| \gtrsim I(E)| \gtrsim \left[ \left| \aff(X_{\kappa_n,1},\ldots,X_{\kappa_n,n} ) \right|_p \right]^\frac{Q}{d} \gtrsim (\aff(E))^{\frac{Q}{d}} \]
uniformly in $K$ and $E$. This is exactly Part \ref{ocineq} of Theorem \ref{mainthm}.

\subsection{Proof of Part \ref{nontriv} of Theorem \ref{mainthm}}

The final piece of Theorem \ref{mainthm} is to show that $\alpha = d/Q$ is a nontrivial exponent in the sense that there is always some submanifold $\mathcal M$ of dimension $d$ in $\R^n$ for which the Oberlin condition \eqref{oberlin} is satisfied with exponent $\alpha$ for some nonzero measure on $\mathcal M$. In fact, it suffices to consider the case when $\mathcal M$ is essentially $\R^d$ and the immersion $f$ is a polynomial embedding. In principle, one needs only to show that $\aff$ is nonzero in some such case (since the arguments of the previous section apply to show that \eqref{oberlin} holds locally on $\mathcal M$, and then a scaling argument establishes the same result globally). In light of the estimate \eqref{norm} for the value of $\aff$, the existence of submanifolds with nonzero affine measure is exactly equivalent to the existence of submanifolds for which the affine curvature tensor does not belong to the nullcone of the space of $Q$-linear covariant tensors. The nullcone is difficult if not impossible to describe explicitly, and determining whether a tensor of the very special form \eqref{affdef} belongs to it or not turns out to be a significant challenge. The key observation is that it so happens that critical points (as a function of $M$) in the infimum definition \eqref{affdef} of $\aff$ must be points at which the infimum is attained.  This will be the main observation to be exploited; a secondary observation, encapsulated in the following lemma, allows one to simplify the structure of the affine curvature tensor $\mathcal A_p$ at the expense of infimizing over a larger group:
\begin{lemma}
Let $A$ be a real $n \times m$ matrix where $m \geq n$, and let $[A]_{i_1 \cdots i_n}$ be the $n \times n$ matrix formed by combining columns $i_1,\ldots,i_n$ into a square matrix, i.e., the $(j,k)$ entry of this matrix is $A_{ji_k}$. Then
\begin{equation} \sum_{i_1,\ldots,i_n=1}^m |\det [A]_{i_1 \cdots i_n}|^2 = \frac{n!}{n^n} \left[ \inf_{M \in \SL(n,\R)} \sum_{j=1}^n \sum_{i=1}^m \left| \sum_{k=1}^n M_{jk} A_{ki} \right|^2 \right]^n. \label{detnorm} \end{equation}
\end{lemma}
\begin{proof}
First observe that both
\[ A \mapsto \sum_{i_1,\ldots,i_n = 1}^m |\det [A]_{i_1 \cdots i_n} |^2 \mbox{ and } A \mapsto \sum_{j=1}^n \sum_{i=1}^m |A_{ji}|^2 \]
are invariant under the action of $O(n,\R)$ on the columns of $A$ as well as the action of $O(m,\R)$ on the rows of $A$ (the former assertion is relatively simple; the latter case rests on the observation that tensor products of elements of an orthonormal basis generate an orthonormal basis on the space of tensors in a natural way). In particular, this means that we may, by the singular value decomposition, assume without loss of generality that 
\[ A_{ji} = \sigma_j \delta_{ji} \]
where $\sigma_j$ is the $j$-th singular value of $A$. Thus
\[ \sum_{i_1,\ldots,i_n=1}^m |\det [A]_{i_1 \cdots i_n}|^2 = n! \sigma_1^2 \cdots \sigma_n^2 \mbox{ and }  \sum_{j=1}^n \sum_{i=1}^m |A_{ji}|^2 = \sigma_1^2 + \cdots + \sigma_n^2. \]
By the AM-GM inequality,
\begin{equation} \frac{1}{n!} \sum_{i_1,\ldots,i_n=1}^m |\det [A]_{i_1 \cdots i_n}|^2 \leq \left[ \frac{1}{n} \sum_{j=1}^n \sum_{i=1}^m |A_{ji}|^2 \right]^n \label{amgm1} \end{equation}
with equality if and only if the singular values of $A$ are all equal. Now multiplication of $A$ on the left by a matrix $M \in \SL(n,\R)$ preserves the left-hand side but not necessarily the right-hand side; taking an infimum of the right-hand side over all $M$ gives that
\[ \sum_{i_1=1,\ldots,i_n=1}^m |\det [A]_{i_1 \cdots i_n}|^2 \leq \frac{n!}{n^n} \left[ \inf_{M \in \SL(n,\R)} \sum_{j=1}^n \sum_{i=1}^m \left| \sum_{k=1}^n M_{jk} A_{ki} \right|^2 \right]^n. \]
To show equality, assume without loss of generality that $A$ is diagonal in the standard basis of $\R^{n \times m}$ and let $M$ be the diagonal matrix such that $M_{ii} := \sigma_i^{-1} (\sigma_1 \cdots \sigma_n)^{1/n}$ assuming none of the singular values are zero. In this case, $MA$ has all diagonal entries equal, and consequently \eqref{amgm1} holds with equality when $A$ is replaced by $MA$, giving equality in \eqref{detnorm} as well. If, on the other hand, some singular value $\sigma_{i'}$ of $A$ is zero, let $M^{(t)}$ be another diagonal matrix such that $M_{ii}^{(t)} = t$ for all entries $i \neq i'$ and let $M_{i' i'}^{(t)} = t^{-n+1}$. Then for $t > 0$, $M^{(t)} \in SL(n,\R)$ and
\[ \lim_{t \rightarrow 0^+} \left[ \sum_{j=1}^n \sum_{i=1}^m \left| \sum_{k=1}^n M_{jk}^{(t)} A_{ki} \right|^2 \right]^m = \lim_{t \rightarrow 0^+} \left[ \sum_{i \neq i'} t^2 \sigma_i^2 \right]^n = 0 \]
so \eqref{detnorm} holds with equality again in this case as well.
\end{proof}

In showing that for any pair $(d,n)$ with $1 \leq d < n$, there is a $d$-dimensional submanifold of $\R^n$ for which the corresponding measure $\aff$ is not trivial, it is clear from the definition of $\mathcal A$ and the pigeonhole principle that nontriviality of $\aff$ requires that the vectors $\{X^\alpha f(p)\}_{1 \leq |\alpha| < \kappa_n }$ be linearly independent for any system of coordinate vectors $X_1,\ldots,X_n$. Knowing {\it a priori} that this must be the case, it is possible to essentially factor $\mathcal A$ in such a way that only the highest-order behavior of $f$ at $p$ is relevant for purposes of calculation.
To that end, fix any $p$ and let $V_p$ be the $n_V$-dimensional subspace of $\R^n$ spanned by the vectors $X^{\alpha} f(p)$ as $\alpha$ ranges over all multiindices $\alpha$ with $1 \leq |\alpha| < \kappa_{n}$ (where $\kappa_n$ is the highest order of differentiation that one finds in any column of $\mathcal A$), and let $W_p$ be any $n_W$-dimensional subspace of $\R^n$ chosen so that
\[ V_p \cap W_p = \{0\} \mbox{ and } V_p + W_p = \R^n. \]
It is then possible to uniquely and smoothly write $f$ as a sum $f = f_V + f_W$ such that $f_V$ takes values in $V_p$ and $f_W$ takes values in $W_p$. By definition of $V_p$, it must also be the case that, modulo a constant vector, $f_W$ vanishes to order $\kappa_n$ at the point $p \in \M$. Next fix determinant functionals on $V_p$ and $W_p$ compatible with the determinant on $\R^n$, meaning that
\[ \det (v_1 \wedge \cdots \wedge v_{n_V} \wedge w_1 \wedge \cdots \wedge w_{n_W}) = \det_V(v_1 \wedge \cdots \wedge v_{n_V}) \det_W(w_1 \wedge \cdots \wedge w_{n_W}) \]
when $\{v_1,\ldots,v_{n_V}\}$ and $\{w_1,\ldots,w_{n_W}\}$ are bases of $V$ and $W$, respectively. By the multilinearity of the determinant on $\R^n$, the tensor $\mathcal A_p$ factors at $p$ into pieces that depend on $f_V$ and $f_W$ separately, namely
\begin{align*} {\mathcal A}_p & ( (X_\lambda)_{\lambda \in \Lambda_{d,n}} ) = \det_V  \left(X_{(1,1)} f_V(p) \wedge \cdots \wedge X_{(j_*,1)} \cdots X_{(j_*,\kappa_{n}-1)} f_V(p) \right) \\ & \times \det_W\left(X_{(j_*+1,1)} \cdots X_{(j_*+1,\kappa_{n})} f_W(p) \wedge \cdots \wedge  X_{(n,1)} \cdots X_{(n-1,\kappa_{n})} f_W(p)\right).  
\end{align*}
where $j_*$ is the largest index for which $\kappa_{j_*} < \kappa_{n} $. Splitting the index set $\Lambda_{d,n}$ into subsets $\Lambda_V$ and $\Lambda_W$ for those indices which appear in the first and second terms of this factorization, respectively, it follows that the terms in the factorization are themselves tensors (which up to a normalization constant, are defined intrinsically and smoothly in a neighborhood of the chosen point $p$) which will be called  ${\mathcal A}_{V_p} ((X_\lambda)_{\lambda \in \Lambda_V})$ and  ${\mathcal A}_{W_p} ((X_\lambda)_{\lambda \in \Lambda_W})$,
respectively, so that
\begin{equation} {\mathcal A}_p( (X_\lambda)_{\lambda \in \Lambda} ) = {\mathcal A}_{V_p} ((X_\lambda)_{\lambda \in \Lambda_V}) {\mathcal A}_{W_p} ((X_\lambda)_{\lambda \in \Lambda_W}). \label{factorize} \end{equation}

The fundamental consequence of the factorization \eqref{factorize} is that it allows one to fully separate the contributions of the ``lower order'' parts $\mathcal A_{V_p}$ and the ``higher order'' parts $\mathcal A_{W_p}$. In the former case, it turns out that $\mathcal A_{V_p}$ expressed in coordinates with respect to the basis $X_1,\ldots,X_d$ is actually invariant under the representation $\rho_\cdot^X$ defined by \eqref{reptn}. This is because the vector space of differential operators generated by $X^{\alpha}$ for $1 \leq |\alpha| < \kappa_n$ is invariant under the action of $\rho_M^X$, so there must be a matrix $[\rho_M^X]$ which acts on the column space spanned by 
$$X_{(1,1)} f(p),\ldots,X_{(j_*,1)} \cdots X_{(j_*,\kappa_{j_*})} f(p) $$
which is equal to the action of $\rho_M^X$ on this space.
For every $M \in \SL(d,\R)$, the matrix $[\rho_M^X]$ must have determinant of magnitude $1$. This follows by symmetry when $M$ is a diagonal matrix (since $\mathcal A_p((X_\lambda)_{\lambda \in \Lambda_V}$ will necessarily vanish by the pigeonhole principle unless each vector field $X_j$ occurs an equal number of times, i.e., unless the number of $\lambda \in \Lambda_V$ for which $X_\lambda = X_j$ is a constant function of $j$). Likewise $|\det [\rho^X_M] | = 1$ when $M$ is an orthogonal matrix since continuity of the map $M \mapsto |\det [\rho^X_M] |$ together with the identity $|\det [\rho_{MN}^X]| = |\det [\rho_{M}^X]| \cdot |\det [\rho_N^X]|$ shows that if the maximum or minimum values of $|\det [\rho_M^X]|$ as a function on the orthogonal group were different from $1$, they could not be attained (since one could always use the group law to find a new othogonal matrix with strictly greater or smaller absolute determinant). However, any $M \in \SL(d,\R)$ can always be factored as a product of a diagonal and orthogonal matrix, so $|\det [\rho_M^X]| = 1$ must hold in all cases.

Since $\mathcal A_{V_p}$ is invariant under $\rho_M^X$, one needs merely to show that there is some lower-order part $f_V$ for which it is not identically zero. In this case,
\[ f_V(t_1,\ldots,t_d) := (t^\alpha)_{1 \leq |\alpha| < \kappa_n} \]
(where we interpret the coordinates on the right-hand side as being relative to some choice of basis of $V$)
suffices, since with respect to the standard coordinate vectors $\partial_{t_i}$ the matrix of ${\mathcal A}_{V_p}$ is seen to be lower triangular with nonzero diagonal entries.

Thus the problem is now fully reduced to the study of $\mathcal A_{W_p}$. In this case, we will set
\[ f_W(t_1,\ldots,t_d) := (p_1(t),\ldots,p_m(t)) \]
for some polynomials $p_1,\ldots,p_m$ which are homogeneous of degree $\kappa_n$, where $m$ is less than or equal to the dimension of the vector space of all such homogeneous polynomials.

By work of
Richardson and Slodowy \cite{rs1990} (which is the real analogue of ideas introduced by Kempf and Ness \cite{kn1979}) it suffices to show that there is a choice of $p_1,\ldots,p_m$ such that the map
\[ M \mapsto || \rho_M \mathcal A_{W_p} ||^2 \]
has a critical point (where from here forward, $\rho$ and the norm $||\cdot||$ will be taken with respect to the standard coordinates $\partial_{t_1},\ldots,\partial_{t_d}$) since they showed that all critical points are points where the infimum over all $M \in \SL(d,\R)$ is actually attained. Moreover, it suffices to show that such a critical point exists when $M$ is the identity. By \eqref{detnorm}, the problem can be further reduced to showing that the function
\[ (N,M) \mapsto \sum_{k=1}^m \sum_{j_1,\ldots,j_{\kappa_n} = 1}^d \left| \sum_{\ell=1}^m \sum_{i_1,\ldots,i_{\kappa_n}=1}^d N_{\ell k} M_{i_1 j_1} \cdots M_{i_{\kappa_n} j_{\kappa_n}} \partial_{t_{i_1}} \cdots \partial_{t_{i_{\kappa_n}}} p_\ell(t) \right|^2 \]
has a critical point at the identity as a function of $(N,M) \in \SL(m,\R) \times \SL(d,\R)$ for appropriate choice of $p_1,\ldots,p_m$. Differentiating in $N$ at the identity along some $E \in \mathfrak{sl}(m,\R)$ gives that
\[ 2 \sum_{k=1}^m \sum_{\ell=1}^m E_{\ell k} \sum_{j_1,\ldots,j_{\kappa_n}=1}^d  \partial_{t_{j_1}} \cdots \partial_{t_{j_{\kappa_n}}} p_\ell(t) \partial_{t_{j_1}} \cdots \partial_{t_{j_{\kappa_n}}} p_{k}(t) = 0\]
for all traceless $m \times m$ matrices $E$. A similar calculation differentiating $M$ ultimately gives that critical points are those which satisfy the system
\begin{align}
 \sum_{i_1,\ldots,i_{\kappa_n}} \partial_{t_{i_1}} \cdots \partial_{t_{i_{\kappa_n}}} p_\ell(t) \partial_{t_{i_1}} \cdots \partial_{t_{i_{\kappa_n}}} p_{\ell'}(t) & = \lambda_1 \delta_{\ell,\ell'}, \label{firstp} \\
 \sum_{\ell, i_2,\ldots,i_{\kappa_n}} \partial_{t_{j}} \cdots \partial_{t_{i_{\kappa_n}}} p_\ell(t) \partial_{t_{j'}} \cdots \partial_{t_{i_{\kappa_n}}} p_\ell(t) & = \lambda_2 \delta_{j,j'} \label{secondp}
 \end{align}
 for some real numbers $\lambda_1,\lambda_2$ and all indices $j,j',\ell,\ell'$. At any such critical point, ${\mathcal A}_{W_p}$ will be nonzero exactly when the constants $\lambda_1$ and $\lambda_2$ are nonzero. To simplify matters somewhat, observe that for any real homogeneous polynomial $p(t) = \sum_{|\alpha| = k} c_\alpha t^\alpha$ of degree $k$,
 \[ ||p||_{k}^2 := \sum_{i_1,\ldots,i_k} |\partial_{t_{i_1}} \cdots \partial_{t_{i_k}} p(t)|^2 = \sum_{|\alpha| = k} k! \alpha! |c_\alpha|^2 \]
 since $\partial^\beta_t t^\alpha = \alpha! \delta_{\alpha,\beta}$ and for any multiindex $\beta$, there are $k! / \beta!$ ways to write $\partial^\beta_t$ as an iterated derivative $\partial_{t_{i_1}} \cdots \partial_{t_{i_k}}$. By polarization, the norm $||\cdot||_k$ has an immediate corresponding inner product. In this notation, \eqref{firstp} and \eqref{secondp} become
 \begin{equation} \sum_{i=1}^d \ang{\partial_{t_i} p_\ell, \partial_{t_i} p_{\ell'}}_{\kappa_{n}-1} = \lambda_1 \delta_{\ell,\ell'} \mbox{ and } \sum_{\ell=1}^m \ang{\partial_{t_i} p_\ell, \partial_{t_{i'}} p_\ell}_{\kappa_{n}-1} = \lambda_2 \delta_{i,i'}. \label{short} \end{equation}
It is similarly elementary to compute inner products of monomials:
\begin{equation} \ang{\partial_{t_i} t^\alpha, \partial_{t_{i'}} t^\beta}_{\kappa_n - 1} = \alpha_i \beta_{i'} (\kappa_n-1)! (\alpha-e_i)! \delta_{\alpha-e_i,\beta-e_{i'}} \label{ipcalc} \end{equation}
where $e_i$ is the multiindex which is zero except in position $i$, where it equals $1$ (and note that the right-hand side of \eqref{ipcalc} is to be interpreted as zero if $\alpha_i = 0$ or $\beta_{i'} = 0$).

For simplicity, fix $\kappa := \kappa_n$.
To build a nontrivial $f_W$, we will chose each polynomial $p_1,\ldots,p_m$ to have one of two types. The first type is of the form
\[ p_\ell(t) := \frac{t^{\alpha}}{\sqrt{\alpha!}} \]
for some multiindex with $|\alpha| = \kappa$ which is not a pure $\kappa$ power (i.e., $t^\alpha \neq t^\kappa_i$ for any $i$). We impose a compatibility condition that if $p_j(t) = t^{\alpha} / \sqrt{\alpha!}$ for some $j$, then for every cyclic permutation $\alpha'$ of $\alpha$, there is another index $j'$ such that $p_{j'}(t) = t^{\alpha'}/\sqrt{\alpha'!}$. Assuming that most $p_j$ are of this form, we additionally allow for up to $d$ more polynomials which depend only on the pure $\kappa$ power monomials $t_i^\kappa$ as follows.  Suppose that $\{\varphi_j\}_{j=1,\ldots,d}$ is a uniform, normalized tight frame (UNTF) on $\R^{d_0}$ for some $d_0 \leq d$, which means that
\[ \sum_{j=1}^{d} |\ang{v,\varphi_j}|^2 = ||v||^2 \mbox{ for all } v \in \R^{d_0} \mbox{ and } ||\varphi_j||^2 = \frac{d_0}{d}, \ j=1,\ldots,d. \]
Such collections of vectors are guaranteed to exist for any $d_0 \leq d$ (see \cite{gkk2001} for existence; a general algorithm based on Theorem 7 of \cite{kadison2002} which can convert a NTF to a UNTF is also known \cite{hp2004}). With such a UNTF, one may optionally chose to add exactly $d_0$ polynomials to the collection constituting $f_W$ provided these new polynomials have the form
\[ \sum_{j=1}^d \frac{t_j^\kappa}{\sqrt{\kappa!}} \varphi_{j,k} \mbox{ for } k=1,\ldots,d_0, \]
where $\varphi_{j,k}$ is the $k$-th coordinate of $\varphi_j$ in the standard basis.
(Note that these optional UNTF-generated polynomials can be added for at most a single choice of UNTF.)

To verify the first condition of \eqref{short}, notice that when $\ell \neq \ell'$ and one of $\ell$ or $\ell'$ correspond to indices of a monomial-type polynomial, every inner product in the sum must be zero because $\partial_{t_i} p_\ell$ and $\partial_{t_{i}} p_{\ell'}$ have no monomials in common and are consequently orthogonal.  If $\ell = \ell'$ and the polynomial $p_\ell$ is monomial type, then
\[ \sum_{i=1}^d \ang{\partial_{t_i} \frac{t^\alpha}{\sqrt{\alpha!}}, \partial_{t_{i}} \frac{t^\alpha}{\sqrt{\alpha!}}}_{\kappa-1} = \sum_{i=1}^d \frac{(\kappa-1)! \alpha_i^2 (\alpha-e_i)!}{\alpha!} \delta_{\alpha_i > 0} = \kappa!.\]
 If, in the final case, both $l$ and $l'$ arise from UNTF terms, the left-hand side of the first equality of \eqref{short} must equal
 \[ \sum_{i=1}^d \frac{1}{\kappa!} \varphi_{i,\ell} \varphi_{i,\ell'} \ang{ \partial_{t_i} t_i^\kappa, \partial_{t_i} t_i^\kappa}_{\kappa-1} = \sum_{i=1}^d \frac{1}{\kappa!} \varphi_{i,\ell} \varphi_{i,\ell'} (\kappa!)^2 = \kappa! \delta_{\ell,\ell'} \]
 since the $\varphi_j$ are a normalized tight frame.
 
 As for the second condition of \eqref{short}, by \eqref{ipcalc}, the polynomials $p_\ell$ of monomial type have norms that equal
 \[ \ang{\partial_{t_i} \frac{t^\alpha}{\sqrt{\alpha!}}, \partial_{t_{i'}} \frac{t^\alpha}{\sqrt{\alpha!}}}_{\kappa-1} = (\kappa-1)! \alpha! \delta_{i,i'} \delta_{\alpha_i > 0}. \]
 Summing over all monomial-type polynomials gives a matrix (as a function of $i$ and $i'$) which is a multiple of the identity: simply by symmetry, any monomial appearing in the sum also appears with all its cyclic permutations, so all diagonal entries must be equal.  As for the terms of the sum which arise from UNTF polynomials, 
 \[ \ang{ \partial_{t_i} \sum_{j=1}^d \frac{t_j^\kappa}{\sqrt{\kappa!}} \varphi_{j,k}, \partial_{t_{i'}} \sum_{j=1}^d \frac{t_j^\kappa}{\sqrt{\kappa!}} \varphi_{j,k}}_{\kappa-1} = \frac{|\varphi_{i,k}|^2}{\kappa!} \delta_{i,i'} ||t_i^\kappa||_\kappa^2 =  \kappa! \delta_{i,i'} |\varphi_{i,k}|^2, \]
 which again sums to a multiple of the identity since, after the sum, the $i$-th diagonal entry equals $||\varphi_i||^2$.
 
By the results of Richardson and Slodowy \cite{rs1990}, it is possible to find a nondegenerate highest-order part $f_W$ of the embedding $f$ provided that the dimension $m$ of this highest order part corresponds to the cardinality of a collection of polynomials of the type considered above: monomial-type polynomials for a set of monomials excluding pure powers and invariant under cyclic permutations together with $d_0$ UNTF-type polynomials for any $d_0 \in \{0,\ldots,d\}$.
 To see that any integer $m$ between $1$ and the total number of monomials of degree $\kappa$ (inclusive) can admit such a collection, observe that the possible cardinalities of just the collection of monomial-type polynomials range---with gaps, of course---from $0$ up to the total number of monomials minus $d$.  
 The size of any gap (i.e., consecutive values of $m$ which are not cardinalities of an admissible set of monomials) must be strictly less than $d$ for the simple reason that no equivalence class of monomials modulo cyclic permutation has cardinality greater than $d$.  In other words, if any non-pure power polynomials happen not already to belong to the collection, including any such monomial together with its cyclic permutations (which is a total of $d$ or fewer new monomials) will again make a larger admissible set. Since the gaps are size strictly less than $d$ and since $d_0$ can be chosen as desired in $\{d,\ldots,d\}$ combining both types of polynomials leads to a nondegenerate measure $\aff$ for any possible value of $m$ given the dimension $d$.

\section{Appendix: Uniform bounds on the number of solutions of real analytic systems of equations}
\label{appendix}

We finish with a brief discussion of the problem of uniformly bounding the number of nondegenerate solutions to any system of equations that arises in a geometric function system. The precise statement that is needed is that for arbitrary positive integers $d$ and $n$ (no longer retaining their previous definitions)
when $f_1,\ldots,f_d$ are real analytic functions on a neighborhood of the unit cube $[0,1]^n$, then any system of equations $(\Phi_1(x),\ldots,\Phi_n(x)) = (y_1,\ldots,y_n)$ must have bounded nondegenerate multiplicity when the functions $\Phi_i$ are rational functions of the $f_i$ and finitely many derivatives of each $f_i$. In other words, the number of solutions in $[0,1]^n$ at which the Jacobian is nonzero is bounded above by a constant that depends only on the functions $f_i$ and the complexity of the system,  in this case meaning the degrees of the numerators and denominators and the order of the highest derivative of an $f_i$. To see this, let $S$ be the Cartesian product of $\{1,\ldots,d\}$ with the set of multiindices $\alpha := (\alpha_1,\ldots,\alpha_n)$ such that $|\alpha| := \alpha_1 + \cdots + \alpha_n \leq N$.  For any $\beta$ which is a multiindex on $S$ (i.e., a map from $S$ into nonnegative integers), we define $s^\beta := 
\prod_{(j,\alpha) \in S} (s_{j,\alpha})^{\beta_{j,\alpha}}$ for every $s \in \R^S$ in analogy with the usual notation.
Lastly, define $P$ be the Cartesian product of $\{1,\ldots,n\}$ and multiindices $\beta$ of size at most $N$ on the set $S$. We can then define a mapping $F$ from $[0,1]^n \times \R^n \times \R^{P} \times \R^{P} \times \R^{S}$ into $\R^n \times \R^n \times \R^P \times \R^P \times \R^S$ by means of the formula
\begin{align*}
 F(&x,y,p,q,s)  := \\ & \left(
 \left( \sum_{(1,\beta) \in P} (p_{1,\beta} - y_1 q_{1,\beta}) s^\beta, \ldots, \sum_{(n,\beta) \in P} (p_{n,\beta} - y_n q_{n,\beta}) s^\beta \right) \right. 
 ,y,p,q,\\ & \qquad  \left. \vphantom{ \sum_{(1,\beta) \in P}} \{ s_{j,\alpha} - \partial^\alpha f_j(x)\}_{(j,\alpha) \in S} \right).
 \end{align*}
For a given triple $(y_0,p_0,q_0) \in \R^n \times \R^P \times \R^P$ and any positive scalar $C$, nondegenerate solutions of the system
\begin{equation} \frac{\sum_{|\beta| \leq N} (p_0)_{j,\beta} \prod_{(j',\alpha) \in S} (\partial^\alpha f_{j'}(x))^{\beta_{j',\alpha}}}{\sum_{|\beta| \leq N} (q_0)_{j,\beta} \prod_{(j',\alpha) \in S} (\partial^\alpha f_{j'}(x))^{\beta_{j',\alpha}}} = (y_0)_j, \qquad j = 1,\ldots,n, \label{rationalsys} \end{equation}
will also be nondegenerate solutions of the system
\[ F(x,y,p,q,s) = \left( 0, \frac{1}{C} y_0, \frac{1}{C^2} p_0, \frac{1}{C} q_0, 0 \right). \]
Choosing $C$ so that the right-hand always belongs to a fixed neighborhood of the origin with compact closure, we may use the fact that $F$ is itself real analytic in all parameters and so the number of connected components of the fiber $F^{-1}(0,y_0/C,p_0/C^2,q_0/C,0)$ is bounded uniformly in $y_0, p_0,$ and $q_0$ (which holds, in fact, for any analytic-geometric category in the sense of van den Dries and Miller \cite{vddm1996}), which gives exactly the desired property that there is also a uniform bound on the number of isolated solutions of \eqref{rationalsys}. If the functions $f_j$ are all polynomial, B\'{e}zout's Theorem gives a similar global bound on the number of nondegenerate solutions, i.e.,  for all nondegenerate solutions $x \in \R^n$ rather than simply $[0,1]^n$.

\bibliography{mybib}

\end{document}